\documentclass[reqno]{amsart}
\usepackage{fullpage}
\usepackage{todonotes}
\usepackage{graphicx, amsmath, amssymb, amsthm}
\usepackage{newclude}
\usepackage{pdflscape}
\usepackage{geometry}
\usepackage[hidelinks]{hyperref}
\usepackage[all]{xy}
\usepackage{enumitem}

\usepackage{relsize}
\usepackage[bbgreekl]{mathbbol}
\usepackage{amsfonts}
\DeclareSymbolFontAlphabet{\mathbb}{AMSb} %to ensure that the meaning of \mathbb does not change
\DeclareSymbolFontAlphabet{\mathbbl}{bbold}
\newcommand{\prism}{{\mathlarger{\mathbbl{\Delta}}}}

\newcommand{\rmd}{\mathrm d}
\newcommand{\cs}[1]{(\!({#1})\!)}
\newcommand{\dcs}[1]{\left(\!\left({#1}\right)\!\right)}
\newcommand{\<}{\left\langle}
\let\>=\undefined
\newcommand{\>}{\right\rangle}

\DeclareMathOperator{\gr}{gr}

\newcommand{\tops}[2]{\texorpdfstring{#1}{#2}}

\newcommand{\lf}{\left\lfloor}
\newcommand{\rf}{\right\rfloor}
\newcommand{\lc}{\left\lceil}
\newcommand{\rc}{\right\rceil}
\newcommand{\psr}[1]{[\![#1]\!]}

\newcommand{\C}{\mathbf C}
\newcommand{\F}{\mathbf F}

\newcommand{\N}{\mathbf N}
\newcommand{\Q}{\mathbf Q}
\newcommand{\R}{\mathbf R}
\let\sec=\S
\renewcommand{\S}{\mathbf S}
\newcommand{\T}{\mathbf T}
\newcommand{\W}{\mathbf W}
\newcommand{\Z}{\mathbf Z}

\DeclareMathOperator{\coker}{coker}

\newcommand{\tnsr}{\otimes}

\newcommand{\morph}{\mathop{\longrightarrow}\limits}

\renewcommand{\lim}{\mathop{\operatorname*{lim}\limits_{\longleftarrow}}\limits}

\newcommand{\iso}{\overset\sim\longrightarrow}

\newcommand{\isom}{\cong}

\newcommand{\defeq}{\mathrel{:=}}

\DeclareMathOperator{\RO}{RO}
\DeclareMathOperator{\RU}{RU}

\renewcommand{\O}{\mathcal O}

\DeclareMathOperator{\dlog}{dlog}
\newcommand{\xra}{\xrightarrow}
\newcommand{\Zpcyc}{\Z_p^{\mathrm{cyc}}}
\newcommand{\conv}{\Rightarrow}

\newcommand{\bigoplushat}{\mathop{\widehat\bigoplus}\limits}

\newcommand{\K}{{\mathrm{K}}}

\newcommand{\THH}{\mathrm{THH}}
\newcommand{\TR}{\mathrm{TR}}
\newcommand{\TC}{\mathrm{TC}}

\newcommand{\TF}{\mathrm{TF}}

\newcommand{\bigTR}{\mathbf{TR}}

\newcommand{\Nyg}{\mathcal N}

\DeclareMathOperator{\length}{length}

\newcommand{\rog}{\bigstar}
\newcommand{\trog}{\blacklozenge}
\newcommand{\can}{\mathrm{can}}

\DeclareMathOperator{\fib}{fib}

\newcommand{\hsl}[1]{[\![#1]\!]_\lambda}
\newcommand{\hyper}{\~}

\newcommand{\mup}[1]{\ar@/_1em/[u]_-{#1}}
\newcommand{\mdown}[1]{\ar@/_1em/[d]_-{#1}}

\newcommand{\Ainf}{{\mathbf{A}_{\mathrm{inf}}}}
\renewcommand{\~}{\widetilde}
\renewcommand{\th}{\text{th}}

\newtheorem{theorem}{Theorem}[section]
\newtheorem*{theorem*}{Theorem}

\newtheorem{lemma}[theorem]{Lemma}
\newtheorem{corollary}[theorem]{Corollary}
\newtheorem{proposition}[theorem]{Proposition}

\theoremstyle{definition}
\newtheorem{definition}[theorem]{Definition}
\newtheorem{remark}[theorem]{Remark}

\newtheorem{example}[theorem]{Example}
\newtheorem{warning}[theorem]{Warning}

\begin{document}

\title{Floor, ceiling, slopes, and $K$-theory}

\author[Y.~J.~F.~Sulyma]{Yuri~J.~F. Sulyma}
\email{yuri.sulyma@protonmail.com}
% \thanks{Sulyma was supported in part by NSF grant DMS-1564289}

\begin{abstract}
We calculate $\K_*(k[x]/x^e;\Z_p)$ by evaluating the syntomic cohomology $\Z_p(i)(k[x]/x^e)$ introduced by Bhatt-Morrow-Scholze and Bhatt-Scholze. This recovers calculations of Hesselholt-Madsen and Speirs, and generalizes an example of Mathew treating the case $e=2$ and $p>2$. Our main innovation is systematic use of the floor and ceiling functions, which clarifies matters substantially even for $e=2$. We furthermore observe a persistent phenomenon of \emph{slopes}. As an application, we answer some questions of Hesselholt.
\end{abstract}

\maketitle

\newcommand{\df}{W(k)\<\frac{x^d}{\lf d/e\rf!}\>}
\newcommand{\dc}{W(k)\<\frac{x^d}{\Gamma{\lc d/e\rc}}\dlog x\>}

\tableofcontents

% !TEX root=./slopes.tex

\section{Introduction}

Let $k$ be a perfect $\F_p$-algebra, and consider the truncated polynomial algebra $k[x]/x^e$. In this note we evaluate the \emph{syntomic cohomology} $\Z_p(i)(k[x]/x^e)$ introduced by Bhatt-Morrow-Scholze \cite{BMS2} and Bhatt-Scholze \cite{Prismatic}. This generalizes an example of Mathew \cite[\sec10]{MathewBMS} which treats the case $e=2$ and $p>2$ (and works essentially verbatim for $p\nmid e$).

Given an augmented $k$-algebra $R$, define the \emph{reduced} syntomic cohomology $\~{\Z_p(i)}(R)$ by the fiber sequence
\[ \~{\Z_p(i)}(R) \to \Z_p(i)(R) \to \Z_p(i)(k). \]
In our case we have $\~{\Z_p(i)}(R)=\Z_p(i)(R)$ for $i\ne0$, but the reduced formulation is more convenient in some places, and important when considering more general values of $k$ (Remark \ref{rmk:perfd}).

\begin{theorem}
\label{thm:main}
For $i>0$, we have
\[ \~{\Z_p(i)}(k[x]/x^e) = \frac{\W_{ei}(k)}{V_e \W_i(k)}[-1]. \]
\end{theorem}
This allows us to compute the (reduced) $p$-adic $K$-theory $\~\K_*(k[x]/x^e;\Z_p)$, recovering calculations of Hesselholt-Madsen \cite{HMCyclicPolytopes} and Speirs \cite{SpeirsTrunc}.

\begin{corollary}
\label{cor:main}
The reduced $p$-adic $K$-groups of $k[x]/x^e$ are concentrated in odd degrees, given by
\[ \~\K_{2i-1}(k[x]/x^e;\Z_p) = \frac{\W_{ei}(k)}{V_e \W_i(k)}. \]
\end{corollary}

As an application of our methods, we answer some questions of Hesselholt: the functoriality of $\K_*(k[x]/x^e;\Z_p)$ as $e$ varies; a slope pattern emerging from this functoriality; and the multiplicative structure of $\K_*(k[x]/x^e;\F_p)$.

\begin{remark}
\label{rmk:perfd}
Theorem \ref{thm:main} and Corollary \ref{cor:main} continue to hold for $k$ an arbitrary perfectoid ring. There are several ways to see this:

\begin{itemize}
\item Repeat the calculations of \sec\ref{sec:calculation} using $q$-divided powers and the $q$-de Rham complex \cite[\sec16]{Prismatic}, \cite{ScholzeQ}. We encourage the reader to carry out this calculation; there is an amusing twist that appears in the symbol $\{d,e\}$. To be precise, we should assume that $k$ is a $\Zpcyc=\Z_p[\zeta_{p^\infty}]^\wedge_p$-algebra in order to justify the use of the $q$-de Rham complex. However, we expect that this restriction can be removed, at least for $p$ odd, by tensoring with $\Zpcyc$ and using \cite[Proposition 4.8.8]{APC}.

\item Alternatively, this follows (without restrictions on $k$) from the original approach to computing $\~\K(k[x]/x^e;\Z_p)$ via ``poly-representation graded $\TR$''; see \sec\ref{sub:tr}, in particular Remark \ref{rmk:poly-perfd}. The point is that poly-graded $\TR$ behaves (in the range of poly-degrees relevant to $K$-theory calculations) the same way for arbitrary perfectoid rings as for perfect $\F_p$-algebras\footnote{This may seem surprising, since it is not true for $\Z$-graded $\TR$ above degree 0. Unlike the integer degrees, the limits in the poly-degrees stabilize at a finite stage, so we do not require spherical completeness hypotheses.}. This provides a meta-theorem showing that a wide class of $K$-theory calculations generalize verbatim from perfect $\F_p$-algebras to arbitrary perfectoid rings. It would be interesting to work out a precise general statement here, and also to give a proof directly using prismatic cohomology (not mentioning $\TR$).

\item Finally, Riggenbach has given an independent proof of this (and much more) in \cite[Theorem 1.1, Corollary 5.4]{RiggenbachTruncated}. His strategy is in some ways a mix of the two approaches above, also mixing in the quasisyntomic site.
\end{itemize}
\end{remark}

\begin{remark}
Theorem \ref{thm:main} can be written more suggestively as a fiber sequence
\[ \~{\Z_p(i)}(k[x]/x^e) \to \W_i(k) \xra{V_e} \W_{ei}(k). \]
This perspective is closer to the approach taken in \sec\ref{sub:tr} than the one in \sec\ref{sec:calculation}. It would be interesting to generalize this sequence to $\~{\Z_p(i)}$ of more general algebras.
\end{remark}

% \Huge
% \[ \operatorname{Ind}^G_H(X) = G/H_+\land X \]

\begin{remark}
The truncated polynomial algebra $k[x]/x^e$ can also be written as $k[y^{1/e}]/y$. The latter form is more convenient for taking the limit over $e$ to get the ring $k[y^{1/p^\infty}]/y$, which plays an important role in prismatic theory \cite[\sec12.2]{Prismatic}. Beware that when carrying out the $q$-version of our calculation (Remark \ref{rmk:perfd}), this coordinate change is subject to the \emph{twisted} chain rule
\[ y^n \dlog y = [e]_{q^n} x^{en} \dlog x, \]
so one must be careful with the ``$\dlog$'' symbol in $q$-land.
\end{remark}

\subsection{Outline}
In \sec\ref{sec:prelims}, we review Mathew's strategy for computing $\~{\Z_p(i)}(k[x]/x^e)$ and collect the various identities we will need. We separate these out in order to emphasize the brevity of the actual calculation, which we carry out in \sec\ref{sec:calculation}. In \sec\ref{sec:discussion}, we introduce a visual language for thinking about the result, emphasizing the role of \emph{slopes}. We use this to discuss two further perspectives on the calculation:
\begin{itemize}
\item In \sec\ref{sub:witt}, we rewrite the answer obtained in \sec\ref{sec:calculation} into the formulation appearing in Theorem \ref{thm:main}.

\item In \sec\ref{sub:tr}, we review the original approach to computing $\K(k[x]/x^e;\Z_p)$, recasting it in terms of ``slope poly-representations'' and the techniques of \cite[\sec4]{SulSliceTHH}.
\end{itemize}
In \sec\ref{sec:applications}, we apply our methods to answer some questions of Hesselholt. We make some closing remarks in \sec\ref{sec:epilogue}.

Interactive versions of the figures in this paper, including source code, are available at
\begin{center}
  \url{https://ysulyma.github.io/papers/fcsk/}
\end{center}

\subsection{Acknowledgements}
We are grateful to Ben Antieau and Noah Riggenbach for helpful discussions, and to the anonymous referee for catching some errors and suggesting several improvements.

% !TEX root=./slopes.tex
\section{Preliminaries}
\label{sec:prelims}

This section sets up the context and tools needed for the main calculation. In \sec\ref{sub:strategy}, we review Mathew's strategy from \cite{MathewBMS} for computing $\~\K_*(k[x]/x^e;\Z_p)$. In \sec\ref{sub:notation}, we collect notation and identities that will be used throughout the paper.

\subsection{Strategy}
\label{sub:strategy}
Define the \emph{reduced} $K$-theory $\~\K(k[x]/x^e;\Z_p)$ to be the fiber of $\K(k[x]/x^e;\Z_p)\to\K(k;\Z_p)$. Reduced versions of other invariants are defined similarly. The Dundas-Goodwillie-McCarthy theorem \cite{DGM} says that the cyclotomic trace identifies the reduced $K$-theory with the reduced topological cyclic homology,
\[ \~\K(k[x]/x^e;\Z_p) \iso \~\TC(k[x]/x^e;\Z_p), \]
so we will compute the latter. (The $p$-completion of $\TC$ is redundant in our case, but necessary for the perfectoid generalization.)

By Bhatt-Morrow-Scholze \cite[Theorem 1.12(5)]{BMS2}, $\~\TC(R;\Z_p)$ has a ``motivic'' filtration with graded pieces given by shifts of (reduced) \emph{syntomic cohomology} $\~{\Z_p(i)}(R)$:
\[
  \gr^i\~\TC(R;\Z_p) = \~{\Z_p(i)}(R)[2i].
\]
This gives a spectral sequence
\[
  E^1_{i,j} = H^{i-j}\~{\Z_p(i)}(R) \conv \~\TC_{i+j}(R;\Z_p).
\]
In our case, $\~{\Z_p(i)}(k[x]/x^e)$ will turn out to be concentrated in degree 1, so the spectral sequence collapses to
\[
  \~\TC_{2i-1}(k[x]/x^e;\Z_p) = H^1\~{\Z_p(i)}(k[x]/x^e).
\]

By definition, syntomic cohomology fits into a fiber sequence
\[
  \Z_p(i)(R) \to \Nyg^{\ge i}\prism_R\{i\} \xra{\can-\varphi\{i\}} \prism_R\{i\}
\]
where $\prism_R\{i\}$ is the Breuil-Kisin twisted \emph{absolute prismatic cohomology} of $R$, and $\Nyg^{\ge i}$ denotes the \emph{Nygaard filtration}. Note that in \cite{BMS2}, this was defined using Nygaard-completed prismatic cohomology $\hat\prism_R$. However, the two definitions agree by \cite[Lemma 7.22, Proposition 8.20]{BMS2} or \cite[Corollary 5.31]{AMMN}.

When $R$ is an $\F_p$-algebra, $\prism_R$ is identified with the \emph{derived crystalline cohomology} $LW\Omega_R$ \cite[Theorem 9.1]{MathewBMS}. Furthermore, over a perfectoid base (in this case $\F_p$) we may trivialize and thus ignore the Breuil-Kisin twists. Thus, the expression for syntomic cohomology becomes \cite[Construction 10.2]{MathewBMS}
\[ \Z_p(i)(R) \to \Nyg^{\ge i}LW\Omega_R \xra{\can-\varphi/p^i} LW\Omega_R. \]

Suppose $B$ is a $p$-torsionfree, $p$-complete, and quasisyntomic $\delta$-ring with $B/p=R$; in our case, $B=W(k)[x]/x^e$ with $\delta(x)=0$. As explained in \cite[Example 8.11]{MathewBMS}, in this situation there is a canonical identification
\[ LW\Omega_R = L\Omega_B \]
which takes the Nygaard filtration to the tensor product of the Hodge filtration on $L\Omega_B$ and the $p$-adic filtration on $\Z_p$. That is, we have
\[ \{\Nyg^{\ge *}LW\Omega_R\} = \{L\Omega^{\ge*}_B\} \tnsr \{p^{\max(*,0)}\Z_p\} \]
in the $p$-complete filtered derived category.

Finally, let $A\to A/I=B$ be a surjection of $\delta$-rings with $A$ also $p$-torsionfree, $p$-complete, and quasisyntomic; in our case, $A=W(k)[x]^\wedge_p$ and $I=(x^e)$. We must also assume that $A/p$ is Cartier smooth \cite[Definition 4.18]{MathewBMS} and the Frobenius on $A/p$ is flat; these are satisfied in our case, where $A/p=k[x]$. Let $D$ denote the divided power envelope of $I$ in $A$ \cite[Construction 7.14]{MathewBMS}. By \cite[Theorem 7.16]{MathewBMS}, there is a natural isomorphism between $L\Omega_B$ and the $p$-completion of the divided power de Rham complex $D\otimes_A\Omega^\bullet_A$. This isomorphism identifies the Hodge filtration on $L\Omega_B$ with the tensor product of the divided power filtration on $D$ and the naive filtration on $\Omega^\bullet_A$.

\subsection{Notation and identities}
\label{sub:notation}
Let $\cs x=\max(x,0)$. Note the identity
\begin{equation}
  \label{eq:cs-neg}
  \cs x - x = \cs{-x}.
\end{equation}

We write $\bigoplushat$ for the $p$-completion of a direct sum. For example, there are strict inclusions
\[ \Z_p[t] \lneq \Z_p[t]^\wedge_p \lneq \Z_p\psr t: \]
we have $\sum p^i t^i\in \Z_p[t]^\wedge_p \setminus \Z_p[t]$, and $\sum t^i\in\Z_p\psr t\setminus \Z_p[t]^\wedge_p$. In the literature, one sometimes finds the notation $R{\<t\>}$ for $R[t]^\wedge_p$ (or for $R{\left[\frac{t^n}{n!}\right]}$); we will use $R{\<t\>}$ to denote the free $R$-module of rank one on a generator $t$.

Let $\lf-\rf,\lc-\rc\colon\R\to\Z$ (floor and ceiling) denote the right and left adjoints respectively of the inclusion $\Z\to\R$. We will need slightly extended versions of their universal properties that describe their interaction with $<$ as well as with $\le$:
\begin{align}
  \label{eq:ceil-ext}
  x \le n < y &\iff \lc x\rc \le n < \lc y\rc\\
  \label{eq:floor-ext}
  x < n \le y &\iff \lf x\rf < n \le \lf y\rf
\end{align}
For $n\in\N_{>0}$ and $m,x\in\R$, we have the division identities
\begin{align}
  \label{eq:div-floor}
  \lf\frac{\lf x/m\rf}n\rf &= \lf\frac x{mn}\rf\\
  \label{eq:div-ceil}
  \lc\frac{\lc x/m\rc}n\rc &= \lc\frac x{mn}\rc
\end{align}

The combinatorial interpretation of these functions is:
\begin{align}
  \label{eq:comb-floor}
  \lf n/k\rf &= \#\{1,\dotsc,n\}\cap k\N\\
  \label{eq:comb-ceil}
  \lc n/k\rc &= \#\{0,\dotsc,n-1\}\cap k\N\\
  \label{eq:comb-logp}
  \dcs{\lf \log_p(m/j)\rf + 1} &= \#[1,m]\cap\{p^\bullet j\}
\end{align}

In particular, from \eqref{eq:comb-floor} and \eqref{eq:comb-ceil} we see that
\begin{equation}
  \label{eq:floor-ceil}
  \lf\frac{n-1}k\rf = \lc n/k\rc - 1.
\end{equation}
Significantly, this shows that the left-hand side depends only on $\nu=n/k\in\Q$, which is not obvious \emph{a priori}.

We will also need the following lemmas.
\begin{lemma}
\label{lem:leg-scale}
For $\nu\in\Q_{>0}$, we have
\begin{align*}
  v_p(\lf p\nu\rf!) &= \lf\nu\rf + v_p(\lf\nu\rf!)\\
  v_p\Gamma{\lc p\nu\rc} &= (\lc\nu\rc - 1) + v_p\Gamma{\lc\nu\rc}
\end{align*}
\end{lemma}
\begin{proof}
Use Legendre's formula
\[ v_p(n!) = \sum_{r=1}^\infty \lf\frac n{p^r}\rf \]
along with the identities \eqref{eq:div-floor} and \eqref{eq:floor-ceil}.
\end{proof}

\begin{lemma}
\label{lem:logp-floor}
For $\nu\in\R_{>0}$ and $j\in\N_{>0}$, we have
\[\lf\log_p\frac\nu j\rf = \lf\log_p\frac{\lf\nu\rf}j\rf. \]
\end{lemma}
\begin{proof}
The number $s=\lf\log_p\frac\nu m\rf$ is characterized by
\[
  p^s \le \frac\nu m < p^{s+1}.
\]
Multiplying by $m$, applying \eqref{eq:floor-ext}, and dividing by $m$ shows that this is equivalent to
\[
  p^s \le \frac{\lf\nu\rf}m < p^{s+1}.\qedhere
\]
\end{proof}

% !TEX root=./slopes.tex

\section{Calculation}
\label{sec:calculation}
Let $k$ be a perfect $\F_p$-algebra, let $A=W(k)[x]^\wedge_p$, viewed as a $\delta$-ring where $\delta(x)=0$, let $B=A/x^e$, and let $R=k[x]/x^e$. As explained in \sec\ref{sub:strategy}, $\prism_R=LW\Omega_R$ is given by the $p$-completed divided power de Rham complex of $A$ along $I=(x^e)$, namely
\[
  W(k)\left[x, \frac{x^{ej}}{j!}\right]^\wedge_p
  \morph^{\rmd}
  W(k)\left[x,\frac{x^{ej}}{j!}\right]^\wedge_p\,\rmd x.
\]
Since we are ultimately interested in computing reduced $K$-theory, we will discard the constants in degree 0 to get the reduced crystalline cohomology $\~{LW\Omega}_R$.

We will rewrite the above complex in three steps.
\begin{enumerate}
\item Reindex by polynomial degree $d$: if $d=ej+r$ with $0\le r<e$, then $j=\lf d/e\rf$. The complex becomes
\[
  \bigoplushat_{d\ge1} \df
  \morph^{\rmd}
  \bigoplushat_{d\ge1} W(k)\<\frac{x^d}{\lf d/e\rf!}\,\rmd x\>.
\]

\item ``Diagonalize'' the differential operator: $\rmd(x^d) = dx^d\dlog x$. The complex splits up as $\smash{\bigoplushat_{d\ge1}}$ of
\[
  \df
  \morph^{\rmd}
  W(k)\<\frac{x^d}{\lf \frac{d-1}e\rf!}\dlog x\>.
\]

\item Using \eqref{eq:floor-ceil}, rewrite $\lf\frac{d-1}e\rf!=(\lc d/e\rf-1)!=\Gamma{\lc d/e\rc}$. The complex becomes $\smash{\bigoplushat_{d\ge1}}$ of
\[
    \df
    \morph^{\rmd}
    \dc.
\]
\end{enumerate}

We next consider the Hodge filtration $\smash{\~{L\Omega}}^{\ge i}_{B/W(k)}$, which we recall is given by the pd-filtration and the naive filtration. Writing $|\cdot|$ for the Hodge degree, we have $|x|=0$ and $|\rmd x|=1$, and thus $|{\dlog x}|=1$, and thus $\left|\frac{x^d}{\Gamma{\lc d/e\rc}}\dlog x\right|=\lc d/e\rc$. It follows that $\smash{\~{L\Omega}}^{\ge i}_{B/W(k)}$ is given by
\[
  \bigoplushat_{\lf d/e\rf\ge i} \df
  \morph
  \bigoplushat_{\lc d/e\rc\ge i} \dc.
\]
Being able to index terms by their Hodge degree is why we emphasize the Gamma and ceiling functions rather than writing everything in terms of factorial and floor.

The Nygaard filtration $\Nyg^{\ge i}\~{LW\Omega}_R$, given by the Hodge filtration and the $p$-adic filtration, is $\smash{\bigoplushat_{d\ge1}}$ of
\[
  p^{\cs{i-\lf d/e\rf}} \df
  \morph^{\rmd}
  p^{\cs{i - \lc d/e \rc}} \dc.
\]

Now we evaluate the cohomology of these complexes. Define $\{d,e\}$ so that
\[
  \rmd\left(\frac{x^d}{\lf d/e\rf!}\right) = \{d,e\}\frac{x^d}{\Gamma{\lc d/e\rc}}\dlog x;
\]
explicitly, we have
\[
  \{d,e\} =
  d\frac{\Gamma{\lc d/e\rc}}{\lf d/e\rf!}
  =
  \begin{cases}
    d & e\nmid d\\
    e & e\mid d.
  \end{cases}
\]
Also let
\begin{align*}
  \epsilon(i, d/e)
  &= \cs{i-\lf d/e\rf}-\cs{i-\lc d/e\rc}\\
  &= \begin{cases}1 & \lf d/e \rf < \lc d/e\rc \le i\\0 & \text{else}\end{cases}
\end{align*}
We see that the cohomology of $\~{LW\Omega}_R$ and $\Nyg^{\ge i}\~{LW\Omega}_R$ is concentrated in degree 1, given by
\begin{align*}
  H^1(\~{LW\Omega}_R)_d &=
  W(k)/\{d,e\}\<\dfrac {x^d}{\Gamma{\lc d/e\rc}}\dlog x\>\\
  H^1(\Nyg^{\ge i} \~{LW\Omega}_R)_d
  &=
  W(k)/p^{\epsilon(i, d/e)}\{d,e\}
  \<p^{\cs{i-\lc d/e\rc}} \frac{x^d}{\Gamma{\lc d/e\rc}}\dlog x\>.
\end{align*}

Next we need to understand the canonical and divided Frobenius maps. We claim that
\begin{align*}
  \can\Big(H^1(\Nyg^{\ge i}\~{LW\Omega}_R)_d\Big) &= p^{\cs{i-\lc d/e\rc}}H^1(\~{LW\Omega}_R)_{d}\\
  \varphi/p^i\Big(H^1(\Nyg^{\ge i}\~{LW\Omega}_R)_d\Big) &= p^{\cs{\lc d/e\rc-i}}H^1(\~{LW\Omega}_R)_{pd}.
\end{align*}
The first of these is evident. For the second, we have by definition that $\varphi/p^i$ takes
\begin{align*}
  p^{\cs{i-\lc d/e\rc}} \frac{x^d}{\Gamma{\lc d/e\rc}}\dlog x
  \mapsto
  \frac{p^{\cs{i-\lc d/e\rc}}}{p^{i-1}} \frac{x^{pd}}{\Gamma{\lc d/e\rc}}\dlog x
\end{align*}
Using Lemma \ref{lem:leg-scale} and the identity \eqref{eq:cs-neg}, we have
\begin{align*}
  \frac{p^{\cs{i-\lc d/e\rc}}}{p^{i-1}} \frac{\Gamma{\lc pd/e\rc}}{\Gamma{\lc d/e\rc}}
  &= \frac{p^{\cs{i-\lc d/e\rc}}}{p^{i-\lc d/e\rc}}\\
  &= p^{\cs{\lc d/e\rc -i}}
\end{align*}
as claimed.

The actions of $\can$ and $\varphi/p^i$ are depicted in Figure \ref{fig:bands}.  We see (using $p$-completeness) that $\varphi/p^i-\can$ is surjective, and its kernel is indexed by the set $I_p$ of positive integers coprime to $p$.

\begin{proposition}
\label{prop:h1-zpi}
The cohomology of $\~{\Z_p(i)}(R)$ is
\[
  H^1\~{\Z_p(i)}(R)
  \isom
  \bigoplus_{j\in I_p}
  W(k)/\{p^{\dcs{\lf\log_p\frac{ei}j\rf+1}}j, e\}
\]
\end{proposition}
\begin{proof}
Let $s$ be minimal such that $\varphi/p^i$ with domain $H^1(\Nyg^{\ge i}\~{LW\Omega}_R)_{p^{s+1}j}$ is \emph{not} an isomorphism; these entries are boxed in Figure \ref{fig:bands}.
Then we have
\[
  H^1\~{\Z_p(i)}(R)_j \cong H^1(\Nyg^{\ge i}\~{LW\Omega}_R)_{p^{s+1}j}.
\]
Indeed, we can hit all further images of $\varphi/p^i$ with $\can$, and we can hit all prior images of $\can$ with $\varphi/p^i$.

The condition on $s$ is
\[
  \lc\frac{p^sj}e\rc \le i < \lc\frac{p^{s+1}j}e\rc
\]
Applying \eqref{eq:ceil-ext} and rearranging gives
\[
  s \le \log_p\frac{ei}j < s+1
\]
which shows that $s=\lf\log_p\frac{ei}j\rf$. To finish, note that $\epsilon(i,p^{s+1}j/e)=0$ by definition of $s$.
\end{proof}

\begin{warning}
  The map $\can-\varphi/p^i$ is \emph{not} levelwise surjective as a map of cochain complexes; it manages to be surjective on $H^1$ because $H^1(LW\Omega_{k[x]/x^e})_j=0$ for all $j\in I_p$. Consequently, we cannot model $\Z_p(i)(k[x]/x^e)$ as the \emph{kernel} of $\can-\varphi/p^i$.
\end{warning}

\begin{figure}
\includegraphics[width=\textwidth]{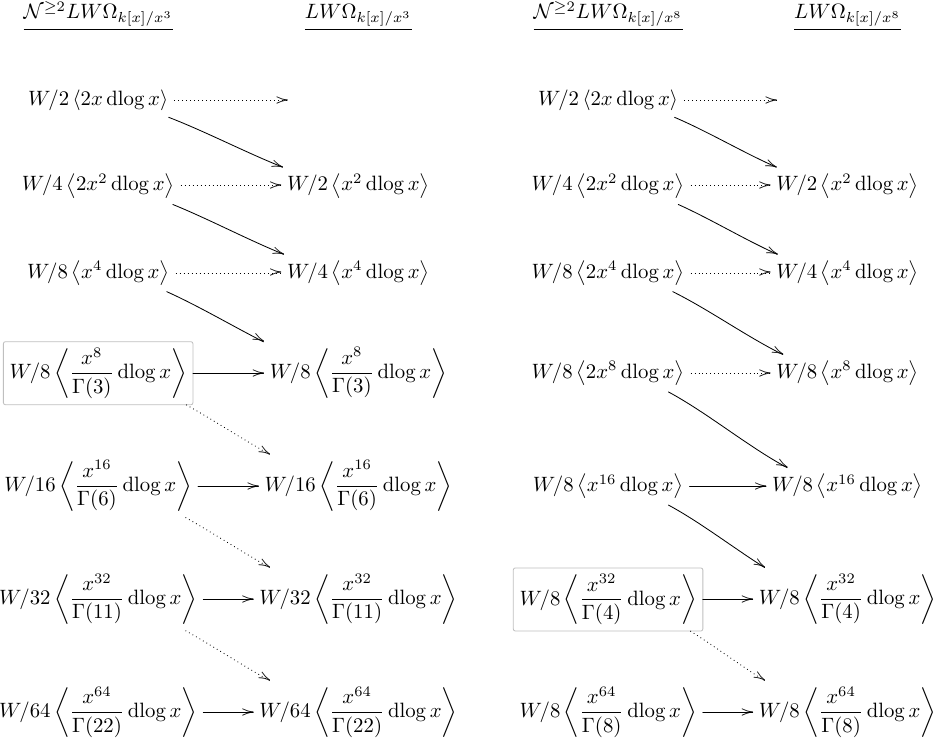}
\caption{Action of $\can$ and $\varphi/p^i$ ($p=2,\, j=1$). Solid lines are isomorphisms, dotted lines are not. The boxes indicate the first entry where $\varphi/p^i$ is \emph{not} an isomorphism.}
\label{fig:bands}
\end{figure}

From Dundas-Goodwillie-McCarthy  and the BMS spectral sequence we have
\[  \~\K_{2i-1}(k[x]/x^e;\Z_p) = \~\TC_{2i-1}(k[x]/x^e;\Z_p) = H^1\~{\Z_p(i)}(k[x]/x^e) \]
so this gives the desired $K$-theory of truncated polynomial algebras. The answer is given in a slightly different form in \cite{HMCyclicPolytopes} and \cite{SpeirsTrunc}, but one checks easily that they are equivalent.

\begin{example}
From Figure \ref{fig:bands}, we see that the $j=1$ summand of $\K_3(\F_2[x]/x^8;\Z_2)$ is generated by
\[ 0 + 0 + 4x^4\dlog x + 2x^8\dlog x + x^{16}\dlog x + \frac{x^{32}}{\Gamma(4)}\dlog x + \frac{x^{64}}{2\Gamma(4)}\dlog x + \dotsb \]
\end{example}

% !TEX root=./slopes.tex

% Discussion
\section{Discussion}
\label{sec:discussion}

To understand the preceding calculation, it is helpful to visualize the situation by plotting $e$ on the $x$-axis and $d$ on the $y$-axis, so that the quantity $d/e$ is a \emph{slope}; see Figure \ref{fig:slopes}. We introduce some terminology to facilitate our discussion.

\begin{figure}
\includegraphics[width=0.7\textwidth]{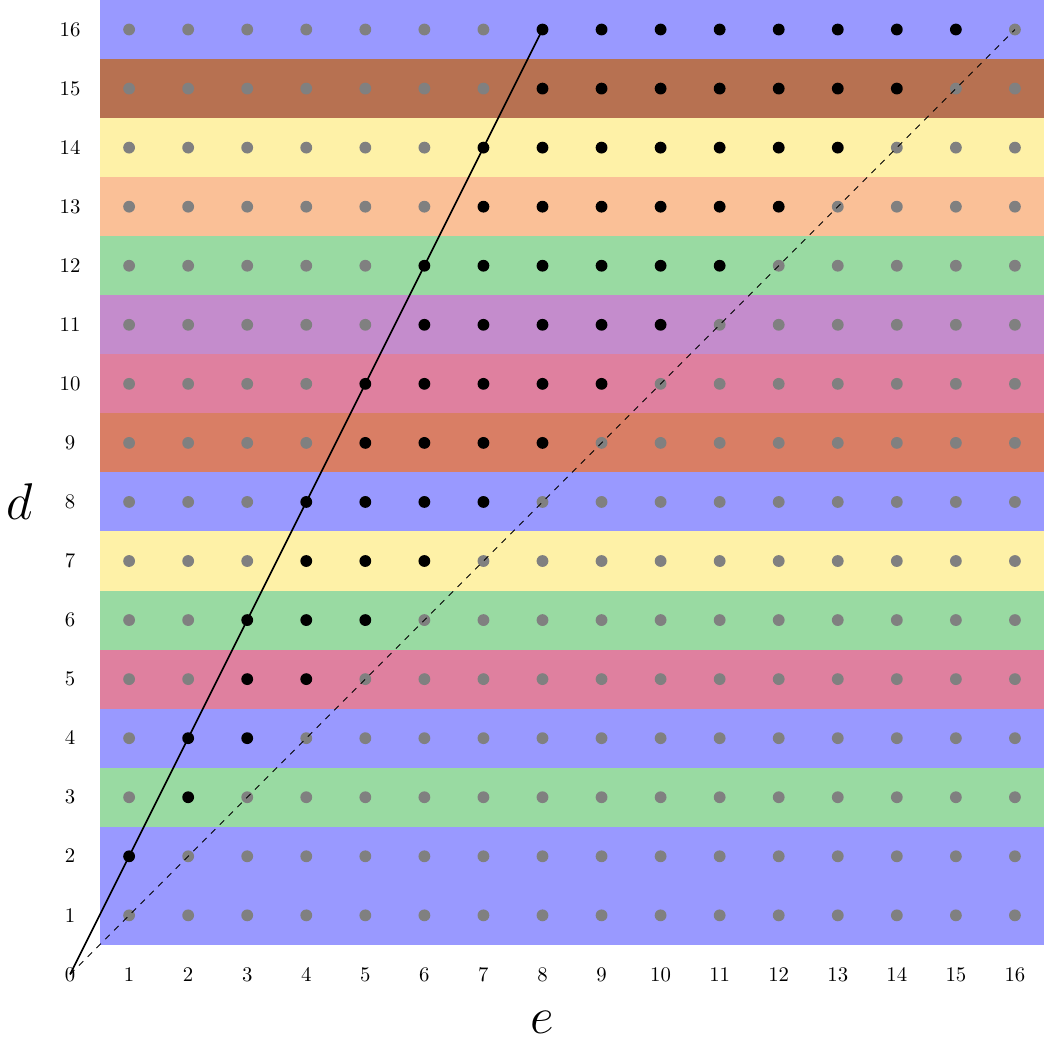}
\caption{Regions bounded by integer slopes}
\label{fig:slopes}
\end{figure}

\begin{definition}
For $j\in I_p$, we call the subset
\[ \{(x,p^a j) \mid x,a\in\N\} \]
the \emph{$j^\text{th}$ $p$-band}. For $i\ge1$, we call the subset
\[ \{(x, ix) \mid x\in\N\} \]
the \emph{$i^\text{th}$ slope ray}.
\end{definition}

With this terminology,
\begin{itemize}
% \item $\lf d/e\rf$ is the last slope ray crossing or below $(e,d)$

\item $\lc d/e\rc$ is the first slope ray crossing or above $(e,d)$

\item $i-\lc d/e\rc$ counts the number of slope rays one crosses traveling vertically from $(e,d)$ to $(e,ei)$, counting the starting point but not the ending point

\item $\lf\log_p\frac{ei}j\rf$ indicates, in the $e^\th$ column, the last entry of the $j^\th$ $p$-band which is on or below the $i^\th$ slope ray
\end{itemize}

If we place $W(k)/\{d,e\}$ in the plane at $(e,d)$, then $\~\K_{2i-1}(k[x]/x^e;\Z_p)$ is given by the sum of the entries in the $e^\th$ column which are the first in their $p$-band to lie strictly above the $i^\th$ slope ray. We next review some alternative interpretations in our language.

% Big Witt vectors
\subsection{Big Witt vectors}
\label{sub:witt}

The result obtained in the last section (Proposition \ref{prop:h1-zpi}) is not yet in the form stated in Theorem \ref{thm:main}, namely
\[ H^1\~{\Z_p(i)}(k[x]/x^e) = \frac{\W_{ei}(k)}{V_e\W_i(k)}. \]
Here $\W_m$ denotes the big Witt vectors with respect to the truncation set of positive integers $\le m$, and $V_e$ the $e^\th$ Verschiebung operator. Recall that a \emph{truncation set} is a subset $J$ of the positive integers closed under divisibility, and that the big Witt vectors $\W_J(k)$ are defined for any such $J$, with $\W_J(k)\isom k^J$ as sets. Moreover, there is a $p$-typical decomposition
\begin{align*}
  \W_J(k)
    &= \prod_{j\in I_p} \W_{J\cap\{j,pj,\dots\}}(k)\\
    &= \prod_{j\in I_p} W_{\#J\cap\{p^\bullet j\}}(k)
\end{align*}
taking $(a_1,a_2,\dotsc)$ to $\{(a_j,a_{pj},a_{p^2j},\dots)\}_j$; this is valid more generally for any $\Z_{(p)}$-algebra \cite[Proposition 1.10]{HesselholtBigW}.

In particular, we can visualize $\W_m(k)$ as the first $m$ entries in any column, and then read off the $p$-typical decomposition from the $p$-bands. Now do this with $\W_{ei}(k)$ in the $e^\th$ column and $\W_i(k)$ in the first, so that both go up to the $i^\th$ slope ray. Then traveling along the slope rays from the first column to the $e^\th$ indicates the image of $V_e\W_i(k)$. Using \eqref{eq:comb-logp}, we see that $\W_{ei}(k)/V_e\W_i(k)$ agrees with $\~\K_{2i-1}(k[x]/x^e;\Z_p)$.

\begin{remark}
Our identification of $\~\K_{2i-1}(k[x]/x^e;\Z_p)$ with $\W_{ei}(k)/V_e\W_i(k)$ comes \emph{a posteriori} by identifying both with the product of $p$-typical Witt vectors. This is also the case in \cite{SpeirsTrunc} and \cite{HMCyclicPolytopes}. It would be nice to have a direct argument.
\end{remark}

% Poly-representations
\subsection{Poly-representations and \tops{$\TR$}{TR}}
\label{sub:tr}

The original approach to calculating $\~\K_*(k[x]/x^e;\Z_p)$ proceeds via analysis of the cyclic bar construction of the pointed monoid $\{0,1,x,\dotsc,x^{e-1}\}$, along with the $\RO(\T)$-graded homotopy of $\THH$. We give a brief account of this from our perspective. Since this subsection uses ideas quite different from the rest of the paper, and is mainly of interest for understanding the relation to \cite{SulSliceTHH}, most readers are advised to skip it. We will use the notation of \cite{SulSliceTHH} (assuming familiarity with that paper), and be fairly terse.

Let $e=p^ve'$ with $p\nmid e'$. Then there is an exact sequence \cite[Proposition 8]{HesselholtHandbook}
\begin{equation}
\label{eq:tr}
0 \to \prod_{j\in e'I_p}\lim_R \TR^{n+1-v}_{2i-[d]_\lambda}(k) \to \prod_{j\in I_p}\lim_R \TR^{n+1}_{2i-[d]_\lambda}(k) \to \TC_{2i-1}(k[x]/x^e) \to 0
\end{equation}
where $d=\lc p^nj/e\rc$, and we have used \eqref{eq:floor-ceil} to rewrite the equation from the form it appears in \emph{loc.\ cit.} We caution the reader that we use $\TR$ to denote \emph{$p$-typical} $\TR$, sometimes denoted $\TR(-;p)$ in the literature, hence the product over $j\in I_p$. We will use $\bigTR$ to denote ``big'' $\TR$, so that $\TR^{n+1}=\bigTR^{p^n}$ and $\bigTR=\prod_{j\in I_p}\TR$.

The limits require some explanation. Recall that in the $\RO(\T)$-graded context, the Restriction maps go
\[ \TR^{n+1}_\alpha \to \TR^n_{\alpha'} \]
so it is not quite correct to speak of ``$\RO(\T)$-graded $\TR$''.

\begin{definition}
Let
\[
  \RO(\T)_p^\sharp \defeq \lim\left(\dotsb\xra'\RO(C_{p^2})\xra{'}\RO(C_p)\xra{'}\RO(e)\right),
\]
called the ring of ($p$-typical, virtual) \emph{poly-representations} of $\T$. There is an evident notion of integral poly-representation, which we will denote by $\RO(\T)^\sharp$. We say that a poly-representation $\hyper\alpha\in\RO(\T)^\sharp$ is an \emph{actual poly-representation} if each $\hyper\alpha(C_n)$ is an actual representation, and \emph{fixed-point free} if $\hyper\alpha(e)=0$. Any poly-representation is of the form $\hyper\alpha$ or $\hyper\alpha+1$ for a \emph{complex} poly-representation $\hyper\alpha\in\RU(\T)^\sharp$, so we will mostly work with the latter. 

For a $p$-typical poly-representation $\hyper\alpha\in\RO(\T)^\sharp_p$, or integral poly-representation $\hyper\beta\in\RO(\T)^\sharp$, we set
\begin{align*}
  \TR_{\hyper\alpha} &\defeq \lim_R \TR^{n+1}_{\hyper\alpha(C_{p^n})}\\
  \bigTR_{\hyper\beta} &\defeq \lim_R \bigTR^n_{\hyper\beta(C_n)}
\end{align*}
We will use $\TR_\trog$ (resp.\ $\bigTR_\trog$) to denote generic $\RO(\T)^\sharp_p$ (resp.\ $\RO(\T)^\sharp$) grading, in analogy to how $\TF_\rog$ is used to denote generic $\RO(\T)$-grading.
\end{definition}

\begin{remark}
Given a poly-representation $\~\alpha\in\RO(\T)^\sharp_p$ (resp.\ $\~\beta\in\RO(\T)^\sharp$), the representation spheres $S^{\~\alpha(C_{p^n})}$ (resp.\ $S^{\~\beta(C_n)}$) assemble into a $p$-typical (resp.\ integral) \emph{polygonic spectrum} in the sense of Krause-McCandless-Nikolaus \cite{KMNPolygonic}.
\end{remark}

A $p$-typical complex poly-representation $\hyper\alpha\in\RU(\T)^\sharp_p$ corresponds uniquely to a sequence $(\~d_0,\~d_1,\~d_2,\dotsc)$, where $\~d_n=\dim_\C\hyper\alpha(C_{p^n})$. We then have
\[
  d_k(\hyper\alpha(C_{p^n})) = \~d_{n-k}(\hyper\alpha).
\]
In this encoding, $\hyper\alpha$ is an actual poly-representation if and only if $\~d_n(\hyper\alpha)\le \~d_{n+1}(\hyper\alpha)$ for all $n$.

For $i>0$, let $\hyper\lambda_i$ denote the $p$-typical poly-representation $C_{p^n}\mapsto \lambda_{n-i}$ (with the understanding that $\lambda_{n-i}=0$ if $n<i$); colloquially, $\hyper\lambda_i$ ``starts at $C_{p^i}$''. These have dimension-sequence
\[
  \~d_k(\hyper\lambda_n) =
  \begin{cases}
    0 & k < n\\
    1 & k \ge n.
  \end{cases}
\]
We write $a_{\hyper\lambda_i}$ and $u_{\hyper\lambda_i}$ for the $\RU(\T)^\sharp_p$-graded classes specializing to $a_{\lambda_{n-i}}$ and $u_{\lambda_{n-i}}$ respectively. We have the following relations in $\TR^{n+1}_\rog(k)$:
\begin{align*}
  p^ia_{\lambda_{n-i}} &= 0\\
  a_{\lambda_j} u_{\lambda_i} &= p^{j-i} a_{\lambda_i} u_{\lambda_j} \quad j>i.
\end{align*}
The first is the fact that $a_\rog$ kills transfers, the second is the gold relation \cite[Lemma 4.9]{SulSliceTHH}. These translate to the following relations in $\TR_\trog(k)$:
\begin{align*}
  p^ia_{\hyper\lambda_i} &= 0\\
  a_{\hyper\lambda_j} u_{\hyper\lambda_i} &= p^{i-j}a_{\hyper\lambda_i} u_{\hyper\lambda_j} \quad j<i.
\end{align*}

We use $\hyper\lambda_0$ to denote the trivial complex poly-representation, which has $\~d_k(\hyper\lambda_0)=1$ for all $k$. Then any $p$-typical poly-representation $\hyper\alpha\in\RU(\T)^\sharp_p$ has an ``irreducible decomposition''
\[
  \hyper\alpha
  =
  \~d_0(\hyper\alpha)\hyper\lambda_0 +
  \sum_{i>0} (\~d_i(\hyper\alpha)-\~d_{i-1}(\hyper\alpha))\hyper\lambda_i,
\]
where the sum on the right may be infinite. 

The following is a corollary of \cite[Lemma 4.12]{SulSliceTHH}, see also \cite[Proposition 9.1]{HMFinite} and \cite[Theorem 8.3]{ROS1TR}.

\begin{proposition}
\label{prop:tr-hyper}
The portion of $\TR_\trog(k)$ of the form $\trog=*-\hyper\alpha$, with $*\in\Z$ and $\hyper\alpha$ an actual poly-representation, is
\[ \TR_\trog(k) = W(k)[a_{\hyper\lambda_i}, u_{\hyper\lambda_i} \mid i>0] \]
except that we are allowed to take the product of infinitely many $a_{\hyper\lambda_i}$s (with each \emph{individual} $a_{\hyper\lambda_i}$ appearing only finitely many times, of course).

Explicitly, let $\hyper\alpha=\sum_{i>0} k_i\hyper\lambda_i$ be an actual, fixed-point free complex poly-representation. Write
\[
  \hyper\alpha[s,t) = k_s\hyper\lambda_s + \dotsb + k_{t-1}\hyper\lambda_{t-1}.
\]
Then
\[
  \TR_{2i-\hyper\alpha}(k)
  =
  \begin{cases}
    W_{s+1}(k)\<u_{\hyper\alpha[1,s+1)} u_{\hyper\lambda_{s+1}}^{i-\~d_s(\hyper\alpha)} a_{\hyper\lambda_{s+1}}^{\~d_{s+1}(\hyper\alpha)-i} a_{\hyper\alpha[s+2,\infty)}\> & \~d_s(\hyper\alpha)\le i < \~d_{s+1}(\hyper\alpha)\\
    W(k)\<u_{\hyper\alpha}\> & \~d_s(\hyper\alpha)=i\text{ for }s\gg0\\
    0 & \~d_s(\hyper\alpha) < i\text{ for all }s\\
    0 & \~d_0(\hyper\alpha) > i
  \end{cases}
\]
\end{proposition}

We now return to our case of interest. For a slope $\nu\in\Q_{>0}$, we see using \eqref{eq:div-ceil} and \eqref{eq:comb-ceil} that
\[
  C_n\mapsto[\lc n\nu\rc]_\lambda
\]
gives a poly-representation, which we denote by $\hsl\nu\in\RU(\T)^\sharp$. With this notation, we can rewrite \eqref{eq:tr} as
\[
0 \to \prod_{j\in I_p}\TR_{2i-\hsl j}(k) \xra{V_e} \prod_{j\in I_p}\TR_{2i-\hsl{j/e}}(k) \to \TC_{2i-1}(k[x]/x^e) \to 0
\]
or more succinctly
\[
0 \to \bigTR_{2i-\hsl1}(k) \xra{V_e} \bigTR_{2i-\hsl{1/e}}(k) \to \TC_{2i-1}(k[x]/x^e) \to 0.
\]
\begin{corollary}
\label{cor:slope-perfd}
For $i\ge1$ and a slope $\nu\in\Q_{>0}$, we have
\begin{align*}
  \TR_{2i-\hsl\nu}(k) &= W_{\dcs{\lf\log_p(i/\nu)\rf+1}}(k),\\
  \bigTR_{2i-\hsl\nu}(k) &= \W_{\dcs{\lf i/\nu\rf}}(k).
  \intertext{In particular,}
  \bigTR_{2i-\hsl{1/e}}(k) &= \W_{ei}(k).
\end{align*}
\end{corollary}
\begin{proof}
By Proposition \ref{prop:tr-hyper}\footnote{The poly-representation $\hsl\nu$ is not fixed-point free. But this assumption is used only to name the generator; it does not affect the criterion for the $W(k)$-module structure.}, we have $\TR_{2i-\hsl\nu}(k)=W_{s+1}(k)$, where $s$ is such that
\[ \~d_s(\hsl\nu)\le i<\~d_{s+1}(\hsl\nu), \]
and $\TR_{2i-\hsl\nu}(k)=0$ if no such $s\ge0$ exists. As a $p$-typical poly-representation, the dimension-sequence of $\hsl\nu$ is $\~d_s(\hsl\nu)=\lc p^s\nu\rc$, so this becomes
\[ \lc p^s\nu\rc \le i < \lc p^{s+1}\nu \rc \]
which by \eqref{eq:ceil-ext} is equivalent to
\[ p^s\nu \le i < p^{s+1}\nu. \]
Rearranging gives
\[ s\le \log_p(i/\nu) < s+1 \]
which says that $s=\lf\log_p(i/\nu)\rf$. This shows that $\TR_{2i-\hsl\nu}(k)$ is as claimed. To deduce $\bigTR_{2i-\hsl\nu}(k)$ from this, we must show that
\[ \lf\log_p\frac i{j\nu}\rf = \lf\log_p\frac {\lf i/\nu\rf}j\rf \]
(c.f.\ the discussion in \sec\ref{sub:witt}), which is a special case of Lemma \ref{lem:logp-floor}.
\end{proof}

\begin{remark}
  \label{rmk:poly-perfd}
  For posterity, we record the analogue of Proposition \ref{prop:tr-hyper} for perfectoid rings; again, this is merely a restatement of \cite[Lemma 4.12]{SulSliceTHH}. Let $S$ be a perfectoid ring, and let $A=\Ainf(S)$. Recall the maps $\theta_i,\,\~\theta_i$ and the elements $\xi_i,\~\xi_i$ from \cite[Lemma 3.12]{BMS1}, as well as the element $\mu$ from \cite[Proposition 3.17]{BMS1}.
  
  We first explain how to translate \cite{SulSliceTHH} to the present context. In that paper we were concerned with $\TF(S;\Z_p)$, and thus viewed $W_n(S)$ as an $A$-algebra via the map
  \[ \phi^{-1}\~\theta_n\colon A/\phi^{-1}(\~\xi_n) \iso W_n(S). \]
  When considering $\TR(S;\Z_p)$, we should instead view $W_n(S)$ as an $A$-algebra via the map
  \[ \theta_n\colon A/\xi_n \iso W_n(S).\]
  Thus we have
  \[ \TR^n_*(S;\Z_p) = A/\xi_n[\sigma_n] \]
  for some choice of generator $\sigma_n\in\TR^n_2(S;\Z_p)$.

  The following Lewis diagram shows how to translate the $F$ and $V$ maps between the two normalizations:
  \[\xymatrix{
    W_3(S) \mdown F & A/\xi\phi(\xi)\phi^2(\xi) \mdown1 \ar[r]^-{\phi^{-2}} & A/\xi\phi^{-1}\phi^{-2}(\xi) \mdown\phi\\
    W_2(S) \mdown F \mup V & A/\xi\phi(\xi) \mdown1 \mup{\phi^2(\xi)} \ar[r]^-{\phi^{-1}} & A/\xi\phi^{-1}(\xi) \mdown\phi \mup{\xi\phi^{-1}}\\
    W_1(S) \mup V & A/\xi \mup{\phi(\xi)} \ar@{=}[r] & A/\xi \mup{\xi\phi^{-1}}
  }\]
  The $R$ map depends on which degree $\TR^\bullet_{2i}(S;\Z_p)$ we are in. The conversion is as follows:
  \[\xymatrix{
    W_3(S){\<\sigma_3^i\>} \ar[d]_-R & A/\xi\phi(\xi)\phi^2(\xi) \ar[d]_-{\phi^{-1}(\xi)^i\phi^{-1}} \ar[r]^-{\phi^{-2}} & A/\xi\phi^{-1}\phi^{-2}(\xi) \ar[d]^-{\phi^{-2}(\xi)^i}\\
    W_2(S){\<\sigma_2^i\>} \ar[d]_-R & A/\xi\phi(\xi) \ar[d]_-{\phi^{-1}(\xi)^i\phi^{-1}} \ar[r]^-{\phi^{-1}} & A/\xi\phi^{-1}(\xi) \ar[d]^-{\phi^{-1}(\xi)^i} \\
    W_1(S){\<\sigma_1^i\>} & A/\xi \ar@{=}[r] & A/\xi
  }\]
  See also \cite[Lemma 3.4]{BMS1}.
  
  Now back to poly-representations. We restrict to degrees $\trog=*-\~\alpha$, with $*\in\Z$ and $\~\alpha$ an actual, fixed-point free poly-representation, and require $*=0$ if $\~\alpha=0$. Translating \cite[Lemma 4.12]{SulSliceTHH}, the corresponding portion of $\TR_\trog(S;\Z_p)$ is
  \[ \TR_\trog(S;\Z_p) = W(S)[a_{\~\lambda_i}, u_{\~\lambda_i}] \]
  subject to the ``Euler'' and ``$q$-gold'' relations \cite[Lemma 4.9]{SulSliceTHH}
  \begin{align*}
    \xi_i a_{\hyper\lambda_{i}} &= 0\\
    a_{\hyper\lambda_{j}} u_{\hyper\lambda_{i}} &= \phi^{-j}(\xi_{i-j}) a_{\hyper\lambda_{i}} u_{\hyper\lambda_{j}} \quad j<i.
  \end{align*}
  In particular, Corollary \ref{cor:slope-perfd} holds for perfectoid rings (provided we $p$-complete $\TR$ and $\bigTR$).

  We can say more in the spherically complete case. Let $C$ be a spherically complete, algebraically closed extension of $\Q_p$, and let $\O_C$ be the ring of integers. In this case we have
  \[ \TR_*(\O_C;\Z_p)=W(\O_C)[\beta],\]
  where $\beta$ is the Bott element \cite[Lemma 7.22]{MathewK1}. Note that $\beta$ maps to $\phi^{-n}(\mu)\sigma_n$ in $\TR^n_2(\O_C;\Z_p)$. The portion of $\TR_\trog(\O_C;\Z_p)$ of the form $\trog=*-\hyper\alpha$, with $*\in\Z$ and $\hyper\alpha$ an actual poly-representation (now allowing $\~\alpha=0$), is thus
  \[ \TR_\trog(\O_C;\Z_p) = W(\O_C)[\beta, a_{\hyper\lambda_{i}}, u_{\hyper\lambda_{i}}] \]
  subject to the above relations, and the additional relation
  \begin{align*}
   \beta a_{\hyper\lambda_{i}} &= \phi^{-i}(\mu) u_{\hyper\lambda_{i}}.
  \end{align*}
\end{remark}

% !TEX root=./slopes.tex

% Applications
\section{Applications}
\label{sec:applications}

We now answer some questions of Hesselholt. In \sec\ref{subsec:func} we identify the induced maps between $\K_*(k[x]/x^e;\Z_p)$ as $e$ varies; the results here are not new, but we give easier proofs. In \sec\ref{subsec:lock}, we (partially) explain the ``interlocking slopes'' phenomenon observed in \cite{LarsTowerGraphics}. In \sec\ref{subsec:mult}, we identify the multiplicative structure of $\K_*(k[x]/x^e;\F_p)$; as far as we know this result is new.

% Functoriality
\subsection{Functoriality}
\label{subsec:func}

Let $\iota_f\colon k[x]/x^e \to k[x]/x^{ef}$ send $x$ to $x^{f}$, and for $m>n$ let $\pi\colon k[x]/x^m\to k[x]/x^n$ be the natural projection. These induce maps on $K$-theory spectra
\begin{align*}
  \pi^* &\colon \K(k[x]/x^m) \to \K(k[x]/x^n)\\
  \iota_{f*} &\colon \K(k[x]/x^{ef}) \to \K(k[x]/x^e)\\
  \iota^*_f &\colon \K(k[x]/x^e) \to \K(k[x]/x^{ef})
\end{align*}
related in a similar manner to the restriction, Frobenius, and Verschiebung maps respectively. In \cite[\sec14]{HesselholtHandbook}, Hesselholt asks for the value of these maps under the identifications above. Armed with explicit names for the generators, we are able to answer this easily.

First, we have
\[
  \iota_f^*\left(\frac{x^d}{\Gamma{\lc d/e\rc}}\dlog x\right) = f\frac{x^{df}}{\Gamma{\lc df/ef\rc}}\dlog x
\]
so that there are commutative diagrams
\[\xymatrix{
  0 \ar[r] & \W_i(k) \ar[r]^-{V_e} \ar@{=}[d] & \W_{ei}(k) \ar@/^1em/[d]^-{V_f} \ar[r] & \K_{2i-1}(k[x]/x^e;\Z_p) \ar@/^1em/[d]^-{\iota_f^*} \ar[r] & 0\\
  0 \ar[r] & \W_i(k) \ar[r]^-{V_{ef}} & \W_{efi}(k) \ar@/^1em/[u]^-{F_f} \ar[r] & \K_{2i-1}(k[x]/x^{ef};\Z_p) \ar@/^1em/[u]^-{\iota_{f*}} \ar[r] & 0.
}\]
This recovers results of Horiuchi \cite{HoriuchiV}.

The case of $\pi^*$ is more complicated. First, some notation: let
\begin{align*}
  M &= \K_{2i-1}(k[x]/x^m;\Z_p)_j &
  N &= \K_{2i-1}(k[x]/x^n;\Z_p)_j\\
  s &= \lf\log_p\tfrac{mi}j\rf = \length_{W(k)}(M) - 1 &
  t &= \lf\log_p\tfrac{ni}j\rf = \length_{W(k)}(N) - 1
\end{align*}
We wish to identify
\[
  \ell := \length_{W(k)}\coker(M\xra{\pi^*} N).
\]

In \cite{LarsTower}, Hesselholt shows that
\begin{equation}
  \label{eq:lars-ans}
  \ell = \sum_{1\le h<i} \lf\log_p\tfrac{mh}j\rf - \lf\log_p\tfrac{nh}j\rf.
\end{equation}
We will explain how to see this from our perspective.

To compare the generators of $M$ and $N$, we must account for (a) the coefficient of $x^j\dlog x$ coming from the canonical map, and (b) the denominators $\Gamma{\lc j/m\rc}$ and $\Gamma{\lc j/n\rc}$. We see that $\pi^*$ sends a generator of $M$ to $p^{\ell'}$ times a generator of $N$, where
\begin{equation}
  \label{eq:our-ans}
  \ell' = \left(\sum_{r=0}^t i-\lc p^rj/n\rc\right) - \left(\sum_{r=0}^s i-\lc p^rj/m\rc\right) + v_p\frac{\Gamma{\lc j/n\rc}}{\Gamma{\lc j/m\rc}}.
\end{equation}

To relate this to \eqref{eq:lars-ans}, we note that
\[
  \lf\log_p\tfrac{nh}j\rf - \lf\log_p\tfrac{mh}j\rf = \#\{ r \mid nh < p^rj \le mh\}.
\]
In our graphical interpretation, \eqref{eq:lars-ans} is counting, for each integral slope $h<i$, the number of times one crosses the $j^\th$ $p$-band while traveling from $(n,nh)$ to $(m,mh)$ along the $h^\th$ slope ray (counting the ending point but not the starting point). On the other hand, \eqref{eq:our-ans} is counting, for each $j$, the number of slope rays one crosses while traveling vertically from $(n,p^rj)$ to $(n,ni)$ (counting the starting point but not the ending point), minus the same for traveling from $(m,p^rj)$ to $(m,mi)$. Since these are two ways of indexing the same sum, we get
\[ \ell' = \ell + v_p\frac{\Gamma{\lc j/n\rc}}{\Gamma{\lc j/m\rc}}. \]

This equation will be distressing to readers who believe in the consistency of mathematics. In fact, there is no contradiction: any time our answer exceeds Hesselholt's, the map $\pi^*$ is actually zero anyway.

\begin{lemma}
If $v_p\Gamma{\lc j/n\rc} > v_p\Gamma{\lc j/m\rc}$, then $\ell>t$ for every $i>0$.
\end{lemma}
\begin{proof}
  By Legendre's formula, we have
  \begin{align*}
    v_p\Gamma{\lc j/n\rc} &=
    v_p{\lf(j-1)/n\rf!}\\
    &= \sum_{r=1}^\infty \lf\frac{\lf(j-1)/n\rf}{p^r}\rf\\
    &= \sum_{r=1}^\infty \lf\frac{j-1}{np^r}\rf\\
    &= \sum_{r=1}^\infty (\lc j/np^r\rc-1)
  \end{align*}
  This gives
  \begin{equation}
    \label{eq:gamma-n-m}
    v_p\Gamma{\lc j/n\rc} - v_p\Gamma{\lc j/m\rc} =
    \sum_{r=1}^\infty \big(\lc j/np^r\rc - \lc j/mp^r\rc\big).
  \end{equation}

  For a fixed $r$, we have
  \[
    \#\{h \mid nh < p^rj \le mh\} = \lc p^rj/n\rc - \lc p^rj/m\rc
  \]
  and thus
  \begin{equation}
    \label{eq:l-sum}
    \ell \ge \sum_{r=0}^t \big(\lc p^rj/n\rc - \lc p^rj/m\rc\big).% + \sum_{r=t+1}^s (i-\lc p^rj/m\rc).
  \end{equation}

  Now if $v_p\Gamma{\lc j/n\rc} > v_p\Gamma{\lc j/m\rc}$, then from \eqref{eq:gamma-n-m} there is some $u$ and some integer $k$ such that
  \[ \frac j{mp^u} \le k < \frac j{np^u}. \]
  It follows that
  \[ \frac{p^r j}m \le p^{u+r}k < \frac{p^r j}n, \]
  so that every term in \eqref{eq:l-sum} is $\ge1$, showing that $\ell>t$.
\end{proof}

\begin{remark}
\label{rem:Fcrys}
Let $M_j^e$ denote the free $W(k)$-module spanned by $\left\{\dfrac{x^{p^rj}}{\Gamma{\lc p^rj/e\rc}}\dlog x \mid r\in\N\right\}$. This has a natural injective action of Frobenius, making it an ``$F$-crystal of infinite rank''. Then
\begin{align*}
  \length_{W(k)} M_j^e/\varphi^{-1}(p^i M_j^e) &= \sum_{r=0}^\infty \cs{i-\lc p^rj/e\rc}\\
  \intertext{and}
  \length_{k} M^e_j/(\varphi^{-1}(p^i M_j^e), p) &= \lf\log_p\tfrac{ei}j\rf+1.
\end{align*}
\end{remark}

% Interlocking slopes
\subsection{Interlocking slopes}
\label{subsec:lock}
We keep the notation of the previous subsection. One of the main results of \cite{LarsTower} is that the map
\[
  \pi^*\colon \K_{2i-1}(k[x]/x^m;\Z_p) \to \K_{2i-1}(k[x]/x^n;\Z_p)
\]
is zero for $i\gg0$. In the companion paper \cite{LarsTowerGraphics}, Hesselholt gives a graphical illustration of this result, plotting those $(j,i)$ for which the $j^\th$ component of this map is nonzero. He concludes with the remark, ``At present, the author does not understand the nature or the value of the apparent interlocking slopes that are seen in the figures below.'' We will try to shed some light on this.

We reproduce some of these graphs in Figures \ref{fig:tower-2-12-11} and \ref{fig:tower-3-12-11}. The solid red bars indicate that the map $\pi^*$ is nonzero for the given values of $i$ and $j$. The hatched green regions are where $v_p\Gamma{\lc j/n\rc}>v_p\Gamma{\lc j/m\rc}$, forcing $\pi^*$ to vanish per the discussion above. In blue, we have plotted lines of slope $p^r/m$ for varying $r$. To understand why these line up so well with the red bars, observe that \eqref{eq:lars-ans} increases whenever $i$ crosses $p^r j/m$. These are not exact bounds, however, so there is perhaps more to say here.

\begin{figure}
  \begin{center}
    \includegraphics[width=0.75\textwidth]{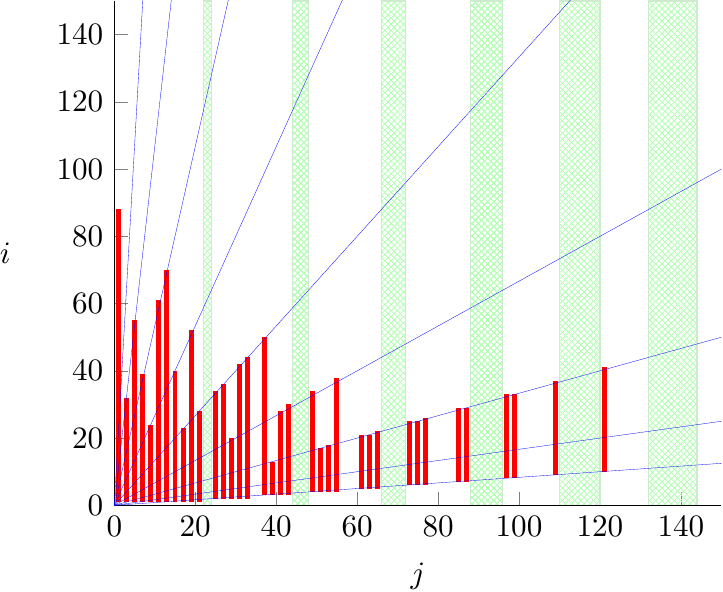}
  \end{center}
  \caption{$m=12$, $n=11$, and $p=2$}
  \label{fig:tower-2-12-11}
\end{figure}

\begin{figure}
  \begin{center}
    \includegraphics[width=0.75\textwidth]{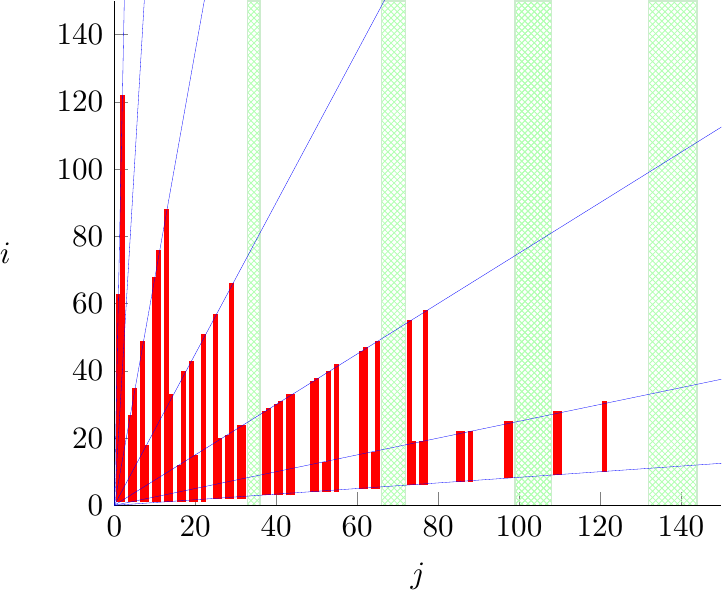}
  \end{center}
  \caption{$m=12$, $n=11$, and $p=3$}
  \label{fig:tower-3-12-11}
\end{figure}

\begin{remark}
Hesselholt's bars go all the way to the bottom of the graph, while ours do not. The ``missing bottoms'' are where the codomain of $\pi^*$ is zero, so this difference does not matter.
\end{remark}

% Multiplication
\subsection{Mod \tops{$p$}p multiplicative structure}
\label{subsec:mult}
Another question raised in \cite[\sec14]{HesselholtHandbook} is the multiplicative structure of $\K_*(k[x]/x^e;\F_p)$. To answer this, we use
\begin{align*}
  \K_{2i}(k[x]/x^e;\F_p) &= H^0\F_p(i)(k[x]/x^e)\\
  \K_{2i-1}(k[x]/x^e;\F_p) &= H^1\F_p(i)(k[x]/x^e),
\end{align*}
where
\[ \F_p(i)(R) \defeq \Z_p(i)(R)/p=\fib(\Nyg^{\ge i}LW\Omega_R/p \xra{\can-\varphi/p^i} LW\Omega_R/p). \]

The quotients by $p$ are of course derived, but since our models for $\Nyg^{\ge i}LW\Omega_R$ and $LW\Omega_R$ are $p$-torsionfree, we can implement this levelwise. Thus $LW\Omega_R/p$ and $\Nyg^{\ge i}LW\Omega_R/p$ are given by $\bigoplus_d$ of
\[
  k\<\frac{x^d}{\lf d/e\rf!}\>
  \xra\rmd
  k\<\frac{x^d}{\Gamma{\lc d/e\rc}}\dlog x\>
\]
and
\[
  k\<p^{\cs{i-\lf d/e\rf}}\frac{x^d}{\lf d/e\rf!}\>
  \xra\rmd
  k\<p^{\cs{i-\lc d/e\rc}}\frac{x^d}{\Gamma{\lc d/e\rc}}\dlog x\>.
\]
Using Lemma \ref{lem:leg-scale} we see that $\varphi/p^i$ takes
\[
  p^{\cs{i-\lf d/e\rf}}W(k)\<\frac{x^d}{\lf d/e\rf!}\>
  \quad\text{to}\quad
  p^{\cs{\lf d/e\rf -i}}W(k)\<\frac{x^{pd}}{\lf pd/e\rf!}\>,
\]
similar to the situation in degree 1.

It follows that $\~\K_*(k[x]/x^e;\F_p)$ is generated by classes $a_i^j$ and $b_i^j$ indexed by $i\ge1$ and $j\in I_p$, with $|a_i^j|=2i$ and $|b_i^j|=2i-1$. The superscripts $j$ should not be mistaken for exponents: rather, they stand for a sequence of exponents $x^{p^\bullet j}$. 
% \begin{align*}
%   a_i^j &= \sum_{r=s_a}^{t_a} p^{\cs{i-\lf p^rj/e\rf}}\frac{x^{p^rj}}{\lf p^rj/e\rf!}\\
%   b_i^j &= \sum_{r=s_b}^{t_b} p^{\cs{i-\lc p^rj/e\rc}}\frac{x^{p^rj}}{\Gamma{\lc p^rj/e\rc}}\dlog x
% \end{align*}
% We get
% \begin{align*}
%   a_{i_1}^{j_1} a_{i_2}^{j_2} &=
%     \sum_{r_1, r_2} 
% \end{align*}
The $b$'s are square-zero, so we need to determine the products
\begin{align*}
  a_{i_1}^{j_1} a_{i_2}^{j_2} &= {?}\\
  a_{i_1}^{j_1} b_{i_2}^{j_2} &= {?}
\end{align*}

\begin{example}
  In $\K_*(\F_2[x]/x^4;\F_2)$, we have classes
  \begin{align*}
    a_2^1 &= px^4 + \frac{x^8}{2!} + \frac{x^{16}}{4!} & b_2^3 &= px^3\dlog x + x^6\dlog x + \frac{x^{12}}{\Gamma(3)}\dlog x\\
    a_2^3 &= px^6 + \frac{x^{12}}{3!} & b_4^5 &= p\dfrac{x^{10}}{\Gamma(3)}\dlog x + \dfrac{x^{20}}{\Gamma(5)}\dlog x\\
    a_4^5 &= p^2\frac{x^{10}}{2!} + \frac{x^{20}}{5!} & b_4^7 &= p^2 x^7\dlog x + \frac{x^{14}}{\Gamma(4)}\dlog x + \frac{x^{28}}{\Gamma(7)}\dlog x\\
    a_4^7 &= p\frac{x^{14}}{3!} + \frac{x^{28}}{7!} & b_4^{11} &= p\frac{x^{11}}{\Gamma(3)}\dlog x + \frac{x^{22}}{\Gamma(6)}\dlog x
  \end{align*}
  In the calculations below, we will fade (some of the) terms that vanish in $H^*(\Nyg^{\ge 4}LW\Omega_{\F_2[x]/x^4}/p)$.
  
  An example of an $aa$ product is
  \begin{align*}
    a_2^1 a_2^3
      &= \sum\left(
        \begin{array}{ccc}
          p^2 x^{10} & p\dfrac{x^{14}}{2!} & \color{gray}p\dfrac{x^{22}}{4!}\\[1em]
          \color{gray} p\dfrac{x^{16}}{3!} & \dfrac{x^{20}}{2!3!} & \dfrac{x^{28}}{3!4!}
        \end{array}
      \right)\\
     &= pa_4^5 + a_4^7\\
     &= a_4^7.
  \end{align*}
  An example of an $ab$ product is
  \begin{align*}
    a_2^1b_2^3
      &= \sum\left(
        \begin{array}{ccc}
          p^{2}x^{7}\dlog x & p\dfrac{x^{11}}{2!}\dlog x & \color{gray} p\dfrac{x^{19}}{4!}\dlog x\\[1em]
          px^{10}\dlog x & \dfrac{x^{14}}{2!}\dlog x & \dfrac{x^{22}}{4!}\dlog x\\[1em]
          \color{gray} p\dfrac{x^{16}}{\Gamma(3)}\dlog x & \dfrac{x^{20}}{2!\Gamma(3)}\dlog x & \dfrac{x^{28}}{4!\Gamma(3)}\dlog x
        \end{array}
      \right)\\[1em]
      &= pb_4^5 + b_4^7 + b_4^{11}\\
      &= b_4^7 + b_4^{11}  .
    \end{align*}
    Beware that
    \begin{align*}
      a_2^1 a_2^1
        &= \sum\left(
          \begin{array}{ccc}
            p^2 x^8 & p\dfrac{x^{12}}{2!} & \color{gray} p\dfrac{x^{20}}{4!}\\[1em]
            p\dfrac{x^{12}}{2!} & \dfrac{x^{16}}{2!2!} & \dfrac{x^{24}}{2!4!}\\[1em]
            \color{gray} p\dfrac{x^{20}}{4!} & \dfrac{x^{24}}{4!2!} & \dfrac{x^{32}}{4!4!}
          \end{array}
        \right)\\
       &= pa_4^1 + 2a_4^3\\
       &= 0,
    \end{align*}
    i.e.\ a term could vanish due to appearing multiple times in the sum.
\end{example}

In general, $a_i^j$ (resp.\ $b_i^j$) is the sum of terms from the last degree (call it $p^sj$) where $\can$ is not an isomorphism, up to and including the first degree (call it $p^tj$) where $\varphi/p^i$ is not an isomorphism. Explicitly:

\begin{proposition}
  The bounds for the generators are
  \begin{align*}
    s_a(i,j) &= \lc\log_p\tfrac{ei}j\rc-1 & s_b(i,j) &= \lf\log_p\tfrac{e(i-1)}j\rf\\
    t_a(i,j) &= \lc\log_p\tfrac{e(i+1)}j\rc & t_b(i,j) &= \lf\log_p\tfrac{ei}j\rf+1
  \end{align*}
\end{proposition}
\begin{proof}
The expression for $t_b(i,j)$ is derived in the proof of Proposition \ref{prop:h1-zpi}. To adapt this to $t_a(i,j)$, the key step is to rewrite
\[
  \lf\frac{p^{t-1}j}e\rf \le i < \lf\frac{p^tj}e\rf
\]
as
\[
  \lf\frac{p^{t-1}j}e\rf < i+1 \le \lf\frac{p^tj}e\rf
\]
in order to apply \eqref{eq:floor-ext}. The others are similar.
\end{proof}

We have not been able to find a concise closed form for these products. However, it is straightforward to evaluate them by computer. We provide some charts in Figures \ref{fig:fp-mult-1}--\ref{fig:fp-mult-4}.

\begin{remark}
  Computation suggests that all $a_{i_1}^{j_1} a_{i_2}^{j_2}$ products vanish when $e=2$ and $p>2$. Proving this could be a good first step towards a more conceptual understanding of the ring structure.
\end{remark}

\newgeometry{margin=0.5in}
\begin{landscape}
  \thispagestyle{empty}

  \vspace*{\fill}
  \begin{figure}[h]
    \includegraphics[width=\linewidth]{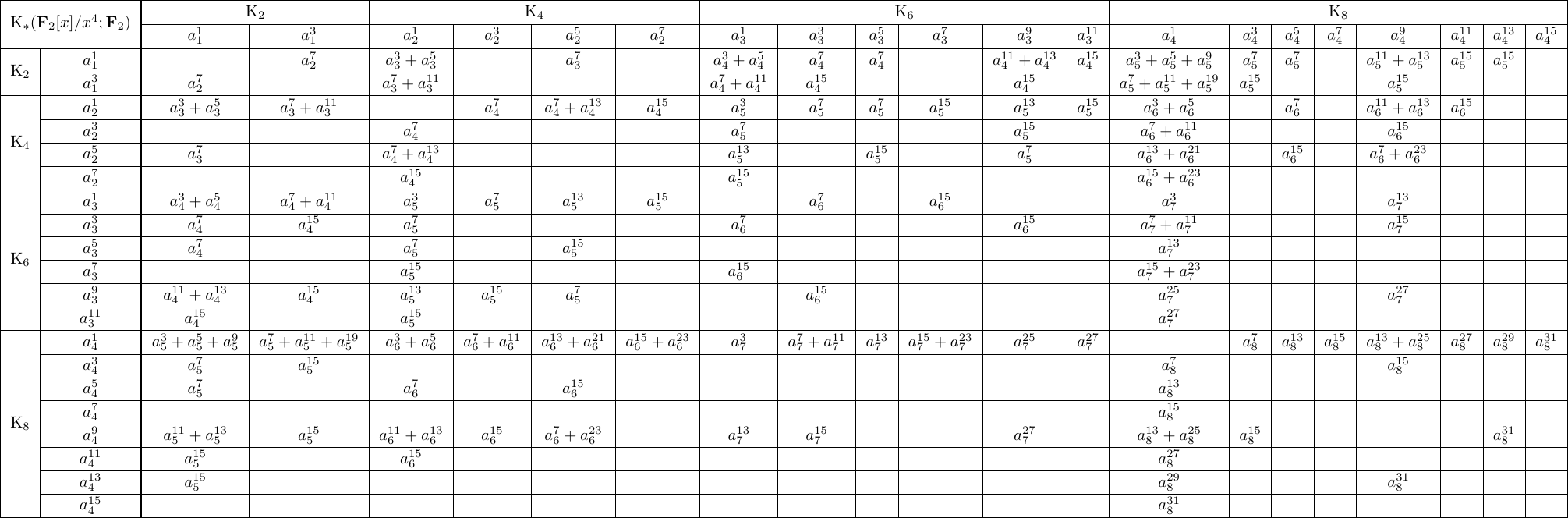}
    \caption{Even-even times table for $\K_*(\F_2[x]/x^4; \F_2)$}
    \label{fig:fp-mult-1}
  \end{figure}

  \vspace*{\fill}

  \begin{figure}[h]
    \includegraphics[width=\linewidth]{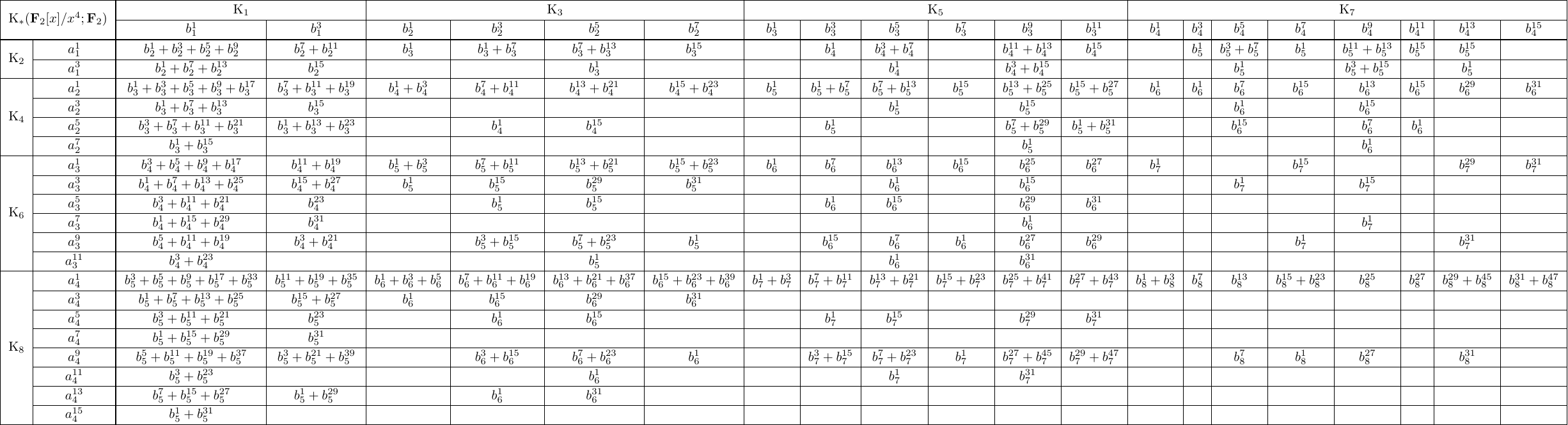}
    \caption{Even-odd times table for $\K_*(\F_2[x]/x^4; \F_2)$}
    \label{fig:fp-mult-2}
  \end{figure}
  \vspace*{\fill}
\end{landscape}
\restoregeometry

\newgeometry{tmargin=.5in,bmargin=.5in,lmargin=.25in,rmargin=.25in}
\begin{landscape}
  \vspace*{\fill}
  \thispagestyle{empty}
  \begin{figure}[h]
    \includegraphics[width=0.75\linewidth]{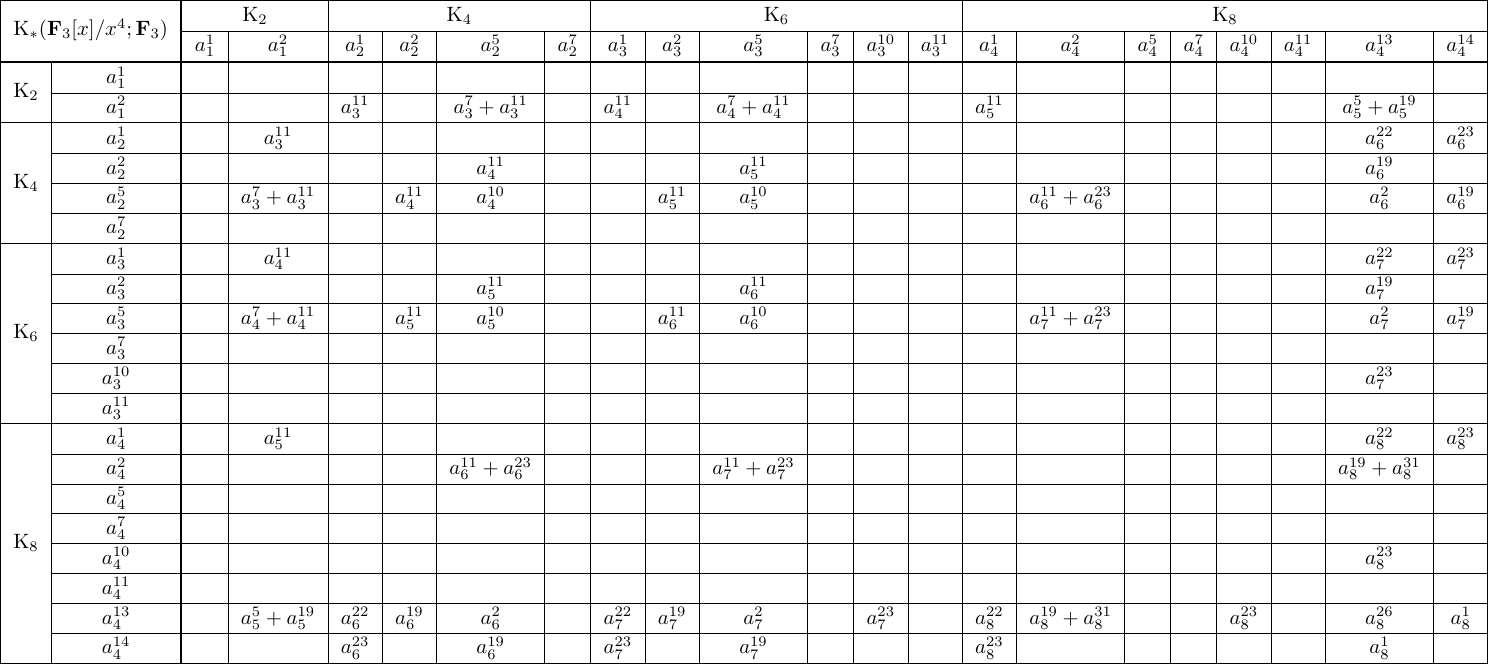}
    \caption{Even-even times table for $\K_*(\F_3[x]/x^4; \F_3)$}
    \label{fig:fp-mult-3}
  \end{figure}
  \vspace*{\fill}
  \begin{figure}[h]
    \includegraphics[width=\linewidth]{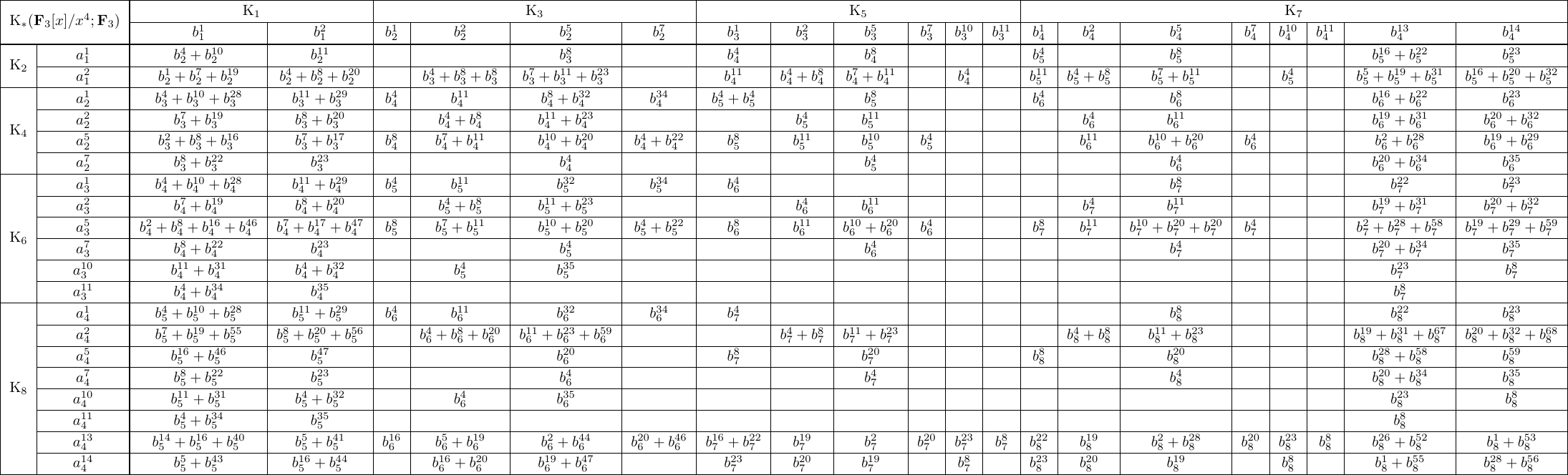}
    \caption{Even-odd times table for $\K_*(\F_3[x]/x^4; \F_3)$}
    \label{fig:fp-mult-4}
  \end{figure}
  \vspace*{\fill}
\end{landscape}
\restoregeometry
% !TEX root=./slopes.tex

% Epilogue
\section{Epilogue}
\label{sec:epilogue}

The $K$-theory of truncated polynomial algebras was originally calculated by Hesselholt-Madsen \cite{HMCyclicPolytopes} using equivariant homotopy theory (along with some difficult results on cyclic polytopes). More recently, Speirs \cite{SpeirsTrunc} simplified the calculation considerably using the Nikolaus-Scholze \cite{NikolausScholze} formula for $\TC$, bypassing equivariant homotopy. The approach above is even simpler, removing homotopy theory altogether. For the purposes of computing algebraic $K$-theory, this is very exciting. For the purposes of tricking arithmetic geometers into doing homotopy theory, it is an alarming trend.

At present, our interpretation in terms of slopes is mainly a visualization aid. It would be extremely interesting to relate this to e.g.\ slopes of $F$-crystals; Remark \ref{rem:Fcrys} is a gesture in this direction.

Our motivation for this computation was to try to understand the similarity between the formulas in \cite{SulSliceTHH} and those of \cite{HMCyclicPolytopes} and \cite{SpeirsTrunc}. We still do not have a satisfactory explanation for this.

The results above, together with those of \cite{SulSliceTHH}, establish a close connection between the following pairs of concepts which are ``off-by-one'' from each other:
\begin{center}
\begin{tabular}{ll}
  ($q$-)factorial & ($q$-)Gamma function\\
  floor & ceiling\\
  reduced regular representation & regular representation\\
  classical slice filtration & regular slice filtration
\end{tabular}
\end{center}

It would be very interesting to extend this table by more columns: take for example $q$-binomial coefficients, $q$-Bernoulli numbers, $q$-Catalan numbers; symmetric and alternating powers of representations; or the generalized slice filtrations of \cite{Wilson}. At the time of writing \cite{SulSliceTHH}, the author believed the classical slice filtration to be a historical misstep, and the regular slice filtration to be the ``correct'' version; writing this paper has led us to reconsider this prejudice. Another interesting question is whether more rows can be added.

% \bibliographystyle{amsalpha}
% \bibliography{../../bibliography}
\providecommand{\bysame}{\leavevmode\hbox to3em{\hrulefill}\thinspace}
\providecommand{\MR}{\relax\ifhmode\unskip\space\fi MR }
\providecommand{\MRhref}[2]{%
  \href{http://www.ams.org/mathscinet-getitem?mr=#1}{#2}
}
\providecommand{\href}[2]{#2}

\end{document}